\documentclass[11pt, a4paper, oneside,leqno]{article}
\usepackage{amsmath, amsfonts, amsthm,dsfont,enumerate}
\newtheorem{mylem}{Lemma}
\newtheorem{mythe}{Theorem}
\newtheorem{myrem}{Remark}
\newtheorem{cor}{Corollary}
\newtheorem{pro}{Proposition}

\title{Generalised cosine functions, basis and \\ regularity properties}

\author{Lyonell Boulton$^1$ \quad \& \quad Houry Melkonian$^2$ \\ \ \\ {\small \it Department of Mathematics and} \\
{\small \it Maxwell Institute for Mathematical Sciences} \\ {\small \it Heriot-Watt University, Edinburgh EH14 4AS, UK.}}
\date{3rd November 2015}

\begin{document}
\maketitle
\footnotetext[1]{Email address: \texttt{L.Boulton@hw.ac.uk}}
\footnotetext[2]{Email address: \texttt{hm189@hw.ac.uk}}

\begin{abstract}
We examine regularity and basis properties of the family of rescaled $p$-cosine functions. We find sharp estimates for their Fourier coefficients. We then determine two thresholds, $p_0<2$ and $p_1>2$, such that this family is a Schauder basis of $L_s(0,1)$ for all   $s>1$ and $p\in[p_0,p_1]$. 
\end{abstract}


\section{Introduction} \label{sec1}
The contents of this paper can be summarised as follows. Consider a continuous 2-periodic function $f:\mathbb{R} \longrightarrow \mathbb{C}$. Denote by $\mathcal{F}$ the family of rescalings $\mathcal{F}=\{f(nx)\}_{n\in \mathbb{N}}$. When does $\mathcal{F}$ form a Schauder basis of $L_s\equiv L_s(0,1)$ for all $s>1$? 
This question can be traced back to a 1945 note by Arne Beurling \cite{Beurling1989}. However,
quite remarkably, there are still a number of open problems associated to it. As it turns, finding a concrete answer can be extremely difficult, even for apparently simple functions $f$.

In a series of recent papers the above question has been addressed for the particular choice $f(x)=\sin_p(\pi_p x)$, the $p$-sine functions. Let $p>1$. Let the increasing function $F_p:[0, 1] \longrightarrow [0,\frac{\pi_p}{2}]$ be defined by means of the integral
\begin{equation} \label{inversepsine}
 F_p(y):=\int_0^y (1-t^p)^{-\frac{1}{p}}\mathrm{d}t
\end{equation}
where
\[
 \pi_p:=2F_p(1)=\frac{2 \pi}{p \sin(\frac{\pi}{p})}.
\]
Denote the inverse of $F_p$ by $\sin_p$, which is increasing in the segment $[0,\frac{\pi_p}{2}]$. Extend to the whole of $\mathbb{R}$ by means of the rules
\begin{equation} \label{sineodd}
    \sin_p(-x)=-\sin_p(x) \qquad \text{and} \qquad \sin_p\left(\frac{\pi_p}{2}-x\right)=\sin_p\left(\frac{\pi_p}{2}+x\right),
\end{equation}
which makes it $2\pi_p$-periodic and continuous in $\mathbb{R}$. The choice $p=2$ corresponds to the standard trigonometric setting $\sin_2\equiv \sin$, $\pi_2=\pi$ and in this case $\mathcal{F}$ is a Schauder basis of $L_s$ for all $s>1$ as a consequence of  Fourier's Theorem.  

The study of the $p$-sine functions originated in the context of the one-dimensional $p$-Laplacian non-linear eigenvalue problem and dates back to the work of Elbert \cite{1979Elbert} and {\^O}tani \cite{1984Otani}. Their basis properties were first examined in \cite{BBCDG2006}, where it was announced that the family $\{\sin_p(n\pi_p\, \cdot)\}_{n\in \mathbb{N}}$ forms a Schauder basis of $L_s$ for all $s>1$ and $p\geq \frac{12}{11}$.  Further development in this respect were  settled in \cite{BushellEdmunds2012}, \cite{EdmundsGurkaLang2012} and \cite{BL2014}. Currently we know that this family is a Schauder basis of $L_s$ for all $s>1$ and $p>\tilde{p}_0$, and also a Riesz basis of $L_2$ for $p\in (\hat{p}_0, \tilde{p}_0]$, where  $\tilde{p}_0\approx 1.087$ and $\hat{p}_0\approx 1.044$ satisfy complicated identities involving hypergeometric functions.

Let 
\begin{equation}\label{cosinedef}
 \cos_p x:=\frac{\mathrm{d}}{\mathrm{d}x}\sin_p x \qquad \forall x \in \mathbb{R}
\end{equation}
and set $f(x)=\cos_p (\pi_p x)$, the $p$-cosine functions. From the various results established in the recent paper \cite{EdmundsGurkaLang2014}, it follows that $\mathcal{F}\cup\{1\}$ is a Schauder basis of $L_s$ for all $s>1$ and $p\in (p_0^\dagger,2]$ where $p^\dagger_0\approx 1.75$. In the present work we establish that this basis property in fact holds true for $p$ in a wider segment. To be precise, we show the following. 

\begin{mythe}\label{melon}
 There exist $p_0<\frac32$ and $p_1>\frac{11}{5}$, such that $\{\cos_p(n\pi_p \,\cdot)\}_{n=0}^\infty$ is a Schauder basis of $L_s$ for all $s>1$
 and $p \in [p_0, p_1]$.
\end{mythe}

The constants $p_0$ and $p_1$ will be given analytically as the zeros of corresponding equations involving the parameter $p$. Their approximated values turn out to be  $p_0\approx 1.46$ and $p_1\approx 2.43$.

The proof of Theorem~\ref{melon} is naturally divided into the cases $1<p<2$ and $p>2$. The different parts of the paper follow this division. In Section~\ref{sec2}
we collect various properties of the $p$-trigonometric functions which will be useful later on. In Section~\ref{sec3} we establish precise upper bounds on the asymptotic behaviour of the Fourier coefficients of $\cos_p(\pi_p\cdot)$. In Section~\ref{sec4} we recall the framework for determining  invertibility of the change of coordinates map 
between the families $\{\cos(n\pi\cdot\}_{n=0}^\infty$ and
$\{\cos_p(n\pi_p\cdot\}_{n=0}^\infty$. 
In Section~\ref{sec5} we assemble the proof of Theorem~\ref{melon},
by combining the crucial criterion \eqref{herb} of Section~\ref{sec4} with the estimates of Section~\ref{sec3}. In the final Section~\ref{sec6} we describe the relation between the results announced here and other existing work.


\section{The generalised trigonometric functions} \label{sec2}

We begin by recalling
 various elementary properties of the $p$-cosine functions. A more complete account on this matter can be found in \cite[Section~2]{BushellEdmunds2012} and \cite[Chapter~2]{EdmundsLang2011}. 

Throughout we shall assume that $1<p<\infty$. Note that $\pi_p$ is a decreasing function, smooth in $p>1$, such that 
\[
 \begin{cases}\pi_p \to \infty & p \to 1^+ \\ \pi_p=\pi & p= 2 \\ \pi_p \to 2 & p\to \infty. \end{cases}  
\]
Here and everywhere below we write $p':=p/(p-1)$. According to \cite[(2.3)]{BushellEdmunds2012}, we know that 
\begin{align}\label{equation62}
 p'\pi_{p'}=p\pi_p.
\end{align}

From \eqref{sineodd} and \eqref{cosinedef} it immediately follows that
$\cos_p$ is $2\pi_p$-periodic,
\[
 \cos_p(x)=\cos_p(-x) \quad \text{and} \quad \cos_p\left(x+\frac{\pi_p}{2}\right)=-\cos_p\left(x-\frac{\pi_p}{2}\right) \qquad \forall x\in \mathbb{R}.
\]
Moreover, setting $y=\sin_p(x)$ for $x\in[0,\pi_p/2]$ in the formula for the derivative of the inverse function of \eqref{inversepsine}, gives
\begin{equation}\label{equation4}
    \cos_p(x)=(1-y^p)^{1/p}=(1-\sin_p(x)^p)^{1/p}.
\end{equation}
Thus, $\cos_p$ is decreasing in $(0,\pi_p/2]$, $\cos_p(0)=1$ and 
$\cos_p(\pi_p/2)=0$. In fact we have, 
\[
\vert \sin_p x \vert ^p+\vert \cos_p x \vert^p=1 \qquad \forall x \in \mathbb{R}.
\]
See \cite[(2.7)]{BushellEdmunds2012}.

\begin{mylem} \label{basicpropcos} For all $x \in [0, \frac1 2)$,
\begin{enumerate}[a.]
 \item \label{lbcp1}
\[
      \cos_p(\pi_p x)=\sin_{p'}\left(\pi_{p'}\left(\frac12 -x\right)\right)^{p'-1}
\]
\item \label{lbcp2} 
\[
\frac{\mathrm{d}}{\mathrm{d}x}\cos_p(x)=-\sin_p(x)^{p-1} \cos_p(x)^{2-p}
\]
\item \label{lbcp3}
\[
\frac{\mathrm{d}^2}{\mathrm{d}x^2}\cos_p(x)=\sin_p(x)^{p-2} \cos_p(x)^{3-2p}[2-p-\cos_p(x)^p] .
\]
\end{enumerate}
\end{mylem}
\begin{proof}
The calculations leading to ``\ref{lbcp1}'' and ``\ref{lbcp2}'' can be found in the proofs of \cite[Proposition~2.2]{BushellEdmunds2012} and \cite[Proposition~2.1]{BushellEdmunds2012}, respectively.
From \eqref{equation4} we get
\begin{align*}
\frac{\mathrm{d^2}}{\mathrm{d}x^2}\cos_p(x)&=(2-p)\sin_p(x)^{2p-2}\cos_p(x)^{3-2p}-(p-1)\sin_p(x)^{p-2}\cos_p(x)^{3-p}\\
&=\sin_p(x)^{p-2}\cos_p(x)^{3-2p}\left[(2-p)\sin_p(x)^p-(p-1)\cos_p(x)^p\right],
\end{align*}
which is ``\ref{lbcp3}''.
\end{proof}

The following inequalities will be important below. 

\begin{mylem}\label{fig} 
Let $1<p \leq q < \infty$ and $x \in [0, \frac12]$. Then 
\begin{enumerate}[a.]
\item \label{figa}
$\sin_p(\pi_p x) \geq \sin_q(\pi_q x)$
\item \label{figb}
$\cos_p(\pi_p x) \leq \cos_q(\pi_q x)$.
\end{enumerate}
 \end{mylem}
\begin{proof}
Statement ``\ref{figa}'' is \cite[Corollary 4.4\,-(iii)]{BushellEdmunds2012}. 

Let us show ``\ref{figb}''. A direct evaluation at $x=0$ and $x=1/2$ gives equality for all $p$ and $q$ at these points, so these two cases are immediate. Let $x \in (0, \frac12)$ be fixed. Since $p'$ is decreasing in $p>1$, 
from part ``\ref{figa}'' it follows that
\[
\frac{\mathrm{d}}{\mathrm{d} p}\sin_{p'}\Big(\pi_{p'} \Big(\frac12-x\Big)\Big)\geq 0 \qquad \forall p\in(1,\infty).
\]
Note that, $0<\sin_{p'}(\pi_{p'} (\frac12-x))<1$ and hence $\ln(\sin_{p'}(\pi_{p'} (\frac12-x)))<0$.
Substituting the identity from Lemma~\ref{basicpropcos}-\ref{lbcp1}, yields
\begin{align*}
&\frac{\mathrm{d}}{\mathrm{d} p}\cos_p(\pi_p x)=\frac{\mathrm{d}}{\mathrm{d} p}\left[\sin_{p'}\Big(\pi_{p'} \Big(\frac12-x\Big)\Big)\right]^{\frac{1}{p-1}}\\
&=\left[-\frac{\ln(\sin_{p'}(\pi_{p'} (\frac12-x)))}{(p-1)^2}+\frac{\frac{\mathrm{d}}{\mathrm{d} p}\left[\sin_{p'}(\pi_{p'} (\frac12-x))\right]}{(p-1)\sin_{p'}(\pi_{p'} (\frac12-x))}\right]\cos_p(\pi_p x)>0.
\end{align*}
This implies ``\ref{figb}''.
\end{proof}

\subsection{The case $1<p<2$}
For $1<p<2$, let $\operatorname{u}_p:[0, \frac1 2]\longrightarrow \mathbb{R}$ be given by
\[
\operatorname{u}_p(x):=\cos_p'(\pi_p x)=
-\sin_p(\pi_p x)^{p-1} \cos_p(\pi_p x)^{2-p}.
\]
This function will simplify the notation when we determine estimates  for the Fourier coefficients of the $p$-cosine functions in Section~\ref{section3.1}. Here and everywhere below we write
 \begin{equation} \label{defcp}
       c_p:=(p-1)^{\frac{p-1}{p}}(2-p)^{\frac{2-p}{p}}.
 \end{equation}

\begin{mylem}\label{Mango}
 Let $1<p<2$. Then
\begin{enumerate}[a.]
 \item\label{Mangoa} $\operatorname{u}_p(x) \leq 0$ for all $x \in [0, \frac1 2]$
 \item\label{Mangob} $\operatorname{u}_p(x)=0$ if and only if $x=0$ or $x=\frac12$
 \item\label{Mangod} $\operatorname{u}_p(x)=-c_p$ for $x\in [0, \frac1 2]$ if and only if $x=m_p\in (0, \frac12)$, where $m_p$ is the unique point such that $\cos_p(\pi_p m_p)^p=2-p$
\item\label{Mangoe} $\operatorname{u}_p:[0, m_p]\longrightarrow [-c_p, 0]$ is decreasing
\item\label{Mangof} $\operatorname{u}_p:[m_p, \frac1 2]\longrightarrow [-c_p, 0]$ is increasing
 \item\label{Mangoc} $\displaystyle\min_{x \in [0, \frac12]}\operatorname{u}_p(x)=-c_p$.
\end{enumerate}
\end{mylem}
\begin{proof}
Since $\sin_p(\pi_p x)$ and $\cos_p(\pi_p x)$ are non-negative over $[0, \frac{1}{2}]$, then ``\ref{Mangoa}'' holds true. Since $\sin_p(\pi_p x)$ only vanishes at $x=0$ and $\cos_p(\pi_p x)$ only vanishes at $x=\frac12$ in this interval, then ``\ref{Mangob}'' holds true.

Lemma~\ref{basicpropcos}-\ref{lbcp3} gives
\[
 \operatorname{u}_p'(x)=\pi_p \sin_p(\pi_p x)^{p-2} \cos_p(\pi_p x)^{3-2p}[2-p-\cos_p(\pi_p x)^p].
\]
Neither $\sin_p$ nor $\cos_p$ vanish in $(0, \frac12)$. On the other hand, $\cos_p(0)=1>2-p$, $\cos_p(\frac{\pi_p}{2})=0<2-p$ and $\cos_p(\pi_p x)^p$ is decreasing for $x \in (0, \frac12)$.
Then the term $\cos_p(\pi_p x)^p+p-2$ indeed vanishes at the unique point $m_p \in (0, \frac12)$ as stated in ``\ref{Mangod}''. 

At $m_p$,
\begin{align*}
      \operatorname{u}_p(m_p)&=-\sin_p(\pi_p m_p)^{p-1} \cos_p(\pi_p m_p)^{2-p}\\
                        &=-(1-\cos_p(\pi_p m_p)^p)^\frac{p-1}{p} \cos_p(\pi_p m_p)^{2-p}=-c_p.
\end{align*} Hence, the proof of ``\ref{Mangoe}'' and ``\ref{Mangof}'', and thus of ``\ref{Mangoc}'', is achieved as follows. Just observe that in the expression for $\operatorname{u}'_p(x)$ above, $\cos_p(\pi_p x)^p>2-p$ for $x \in [0, m_p)$ and $\cos_p(\pi_p x)^p<2-p$ for $x \in (m_p, \frac12)$, because $\cos_p(\pi_p x)$ is decreasing in $x \in (0, \frac12)$. 
\end{proof}

According to parts ``\ref{Mangoe}'' and ``\ref{Mangof}'' of Lemma~\ref{Mango}, the function $\operatorname{u}_p$ is invertible, when restricted to the segments $[0, m_p]$ and $[m_p, \frac12]$. We denote the inverses by $\operatorname{w}_{1,p}: [-c_p, 0] \longrightarrow [0, m_p]$ and $\operatorname{w}_{2,p}: [-c_p, 0] \longrightarrow [m_p, \frac12]$, respectively,
so that 
\[
       \operatorname{u}_p(\operatorname{w}_{k,p}(x))=x \qquad \forall x\in [-c_p, 0] \quad k=1,2.
\]

\subsection{The case $p>2$}
For $p>2$, let
$\operatorname{v}_p:(0, \frac12] \longrightarrow [0, \infty)$ be given by
\[
\operatorname{v}_p(x):=(p'-1)\sin_{p'}(\pi_{p'} x)^{p'-2} \cos_{p'}(\pi_{p'} x).
\] 
Let us summarise various properties of this function, which will be employed in Section~\ref{section3.2}. 

\begin{mylem}\label{almond}
Let $p>2$. Then
\begin{enumerate}[a.]
\item \label{almonda} $\operatorname{v}_p$ is decreasing in $(0, \frac12]$
\item \label{almondb} $\displaystyle\lim_{x \rightarrow 0^+} x \operatorname{v}_p(x)=0$
\item \label{almondc} $\displaystyle \lim_{x \rightarrow 0^+}\operatorname{v}_p(x)=+\infty$ and $\operatorname{v}_p(\frac12)=0$
\item \label{almondd} $\displaystyle \lim_{x\rightarrow 0^+}\operatorname{v}'_p(x)=-\infty$ and $\operatorname{v}'_p(\frac12)=0$.
\end{enumerate}
\end{mylem}
\begin{proof}
For $p>2$, $p'\in (1,2)$ and so $p'-2<0$. Since, $\sin_{p'}(\pi_{p'} x)$ is increasing and $\cos_{p'}(\pi_{p'} x)$ is decreasing in $x\in (0, \frac12)$, then ``\ref{almonda}'' holds true.

Let us show ``\ref{almondb}''. L'H{\^o}pital's Rule gives
 \[
\lim_{x \rightarrow 0^+}\frac{x}{[\sin_{p'}(\pi_{p'} x)]^{2-p'}}=\displaystyle\lim_{x\rightarrow 0^+}\frac{ [\sin_{p'}(\pi_{p'} x)]^{p'-1}}{(2-p')\pi_{p'}\cos_{p'}(\pi_{p'} x)}=0.
 \]
Then, 
\[
\displaystyle\lim_{x \rightarrow 0^+} x \operatorname{v}_p(x)=\displaystyle\lim_{x \rightarrow 0^+}(p'-1)\frac{x\cos_{p'}(\pi_{p'} x)}{[\sin_{p'}(\pi_{p'} x)]^{2-p'}}=0,
\]
as claimed in ``\ref{almondb}''.

Both statements ``\ref{almondc}'' and ``\ref{almondd}'' follow directly from \eqref{equation4} and the expression 
\[
\operatorname{v}'_p(x)=(p'-1)\pi_{p'}\sin_{p'}(\pi_{p'}x)^{p'-3}\cos_{p'}(\pi_{p'}x)^{2-p'}\left[(p'-1)\cos_{p'}(\pi_{p'}x)^{p'}-1\right].
\]
\end{proof}

According to this lemma, there exists a function $\operatorname{z}_p:[0,\infty) \rightarrow (0, \frac12]$ such that $\operatorname{z}_p$ is the inverse function of $\operatorname{v}_p$. This inverse function has the following characteristics.
\begin{enumerate}[a.]
\item $\operatorname{z}_p$ is decreasing in $[0, \infty)$
\item $\operatorname{z}_p(0)=\frac12$ and $\displaystyle\lim_{x \rightarrow \infty}\operatorname{z}_p(x)=0$
\item $\displaystyle\lim_{x \rightarrow 0^+}\operatorname{z}'_p(x)=+\infty$ and $\displaystyle\lim_{x \rightarrow \infty}\operatorname{z}'_p(x)=0$.
\end{enumerate}


\section{The Fourier coefficients of the $p$-cosine functions}
\label{sec3}

Let
\[
 a_j(p)\equiv a_j:=2\int_0^1 \sin_p(\pi_p x) \sin(j \pi x) \mathrm{d}x\qquad \forall j \in \mathbb{N}
\]
be the Fourier sine coefficients of $\sin_p(\pi_p x)$.
Let
\[
 b_j(p)\equiv b_j:=2\int_0^1 \cos_p(\pi_p x) \cos(j \pi x) \mathrm{d}x \qquad \forall j \in \mathbb{N}\cup\{0\}
\]
be the Fourier cosine coefficients of $\cos_p(\pi_p x)$. Since $\sin_p$ is an odd function and $\cos_p$ is an even function, 
$a_j=b_j=0$ for all $j\equiv_2 0$. 

\begin{mylem}\label{plum}
For $j\in \mathbb{N}$,
\begin{align*}
b_j(p)=\frac{j\pi}{\pi_p} a_j(p).
\end{align*}
\end{mylem}
\begin{proof}
Let $j\equiv_2 1$. Integration by parts alongside with the fact that $\cos_p(\pi_p x)$ and $\cos(j \pi x)$ are odd with respect to $\frac{1}{2}$,  yield
\begin{align*}
      b_j&=2\int_0^1 \cos_p(\pi_p x) \cos(j \pi x) \mathrm{d}x=4\int_0^{\frac 1 2} \cos_p(\pi_p x) \cos(j \pi x) \mathrm{d}x\\
         &=\frac{4}{\pi_p}\cos(j \pi x) \sin_p(\pi_p x)\Big\vert_0^\frac12+\frac{4j \pi}{\pi_p}\int_0^\frac1 2 \sin_p(\pi_p x) \sin(j \pi x) \mathrm{d}x\\
         &=\frac{j\pi}{\pi_p} a_j.
\end{align*}
\end{proof}
We now find estimates on $\vert b_j(p) \vert$ in terms of the parameter $p>1$.

\subsection{The case $1<p<2$} \label{section3.1}

\begin{mylem} \label{fourcoeffboundless2}
For $1<p<2$, let $c_p>0$ be given by \eqref{defcp}. Then 
\[
 \vert b_j(p) \vert <\frac{8\pi_p}{j^2 \pi^2}c_p \qquad \forall j \ge 1.
\]
\end{mylem}
\begin{proof} Integrate by parts twice to get
\begin{align*}
b_j&=4\int_0^\frac{1}{2} \cos_p(\pi_p x) \cos(j \pi x) \mathrm{d}x\notag\\
&=\frac{4}{j \pi} \cos_p(\pi_p x) \sin(j \pi x)\Big\vert_0^\frac12-\frac{4\pi_p}{j \pi}\int_0^\frac{1}{2} \cos'_p(\pi_p x) \sin(j \pi x) \mathrm{d}x\notag\\
&=-\frac{4\pi_p}{j \pi}\int_0^\frac{1}{2} \cos'_p(\pi_p x) \sin(j \pi x) \mathrm{d}x \\
&=\frac{4\pi_p}{j^2 \pi^2} \cos'_p(\pi_p x)\cos(j \pi x)\Big\vert_0^\frac12-\frac{4\pi_p}{j^2 \pi^2} \int_0^\frac{1}{2} \frac{d}{\mathrm{d}x} [\cos'_p(\pi_p x)]\cos(j \pi x) \mathrm{d}x.
\end{align*}
From the identities in Lemma~\ref{Mango}-\ref{Mangob}, it follows that the  boundary term in the fourth equality always vanishes. Thus, 
\begin{align*}
b_j&=-\frac{4\pi_p}{j^2 \pi^2} \int_0^\frac{1}{2} \operatorname{u}'_p(x)\cos(j \pi x) \mathrm{d}x\\
&=-\frac{4\pi_p}{j^2 \pi^2}\Big(\int_0^{m_p}+\int_{m_p}^\frac12 \operatorname{u}'_p(x) \cos(j \pi x) \mathrm{d}x\Big) \\
&=-\frac{4\pi_p}{j^2 \pi^2}\Big(\int_0^{-c_p} \cos(j \pi \operatorname{w}_{1, p}(s)) ds+\int_{-c_p}^0 \cos(j \pi \operatorname{w}_{2, p}(s)) \mathrm{d}s\Big).
\end{align*}
Hence, 
\begin{align*}
\vert b_j \vert &\leq \frac{4\pi_p}{j^2 \pi^2}\Big[\int_{-c_p}^0 \vert\cos(j \pi \operatorname{w}_{1, p}(s))\vert ds+\int_{-c_p}^0\vert\cos(j \pi \operatorname{w}_{2, p}(s))\vert \mathrm{d}s\Big]\\
&<\frac{8\pi_p}{j^2 \pi^2}c_p,
\end{align*}
because the functions inside the integrals are not constants identically equal to 1.
\end{proof}

\subsection{The case $p>2$} \label{section3.2}
Let $p>2$. According to Lemma~\ref{basicpropcos}-\ref{lbcp1},  
\begin{align*}
    b_j(p)&= 4 \int _0^{\frac12} 
    \sin_{p'}\left(\pi_{p'}\left(\frac12 -x\right)\right)^{\frac{1}{p-1}} \cos(j\pi x)
    \mathrm{d} x.
\end{align*}
Since $\cos(j \pi(\frac12-t))=(-1)^{\frac{j-1}{2}}\sin(j\pi t)$ for $j\equiv_2 1$, changing variables to $t=\frac12-x$ gives
\begin{align}\label{equation41}
    b_j &= (-1)^{\frac{j-1}{2}}4 \int _0^{\frac12} 
    \sin_{p'}(\pi_{p'} t)^{\frac{1}{p-1}} \sin(j\pi t) \mathrm{d} t\notag.
    \end{align}
By virtue of Lemma~\ref{almond} and integration by parts twice, then
\begin{align}
    b_j& = (-1)^{\frac{j-1}{2}}\frac{4\pi_{p'}}{j\pi} \int _0^{\frac12} 
   \operatorname{v}_p(t) \cos(j\pi t) \mathrm{d} t \notag\\
   &=(-1)^{\frac{j-1}{2}}\frac{4 \pi_{p'}}{j \pi}\left[\frac{1}{j \pi}\operatorname{v}_p(t)\sin(j \pi t)\Big\vert_0^\frac12-\frac{1}{j \pi} \int_0^\frac12 \operatorname{v}'_p(t)\sin(j\pi t)\mathrm{d}t \right]\notag\\
   & = (-1)^{\frac{j+1}{2}} \frac{4\pi_{p'}}{j^2\pi^2} \int _0^{\frac12} 
    \operatorname{v}'_p(t) \sin(j\pi t) \mathrm{d} t \notag\\
   & = (-1)^{\frac{j+3}{2}} \frac{4\pi_{p'}}{j^2\pi^2} \int _0^{\infty} 
    \sin\left(j\pi \operatorname{z}_p(y)\right) \mathrm{d} y .
\end{align}

\begin{mylem}\label{tomato}
Let $p>2$. Then 
\[
    |b_j(p)| < \frac{2\pi_{p'}}{\pi^2(p-1)}\left[2+\frac{\pi^2}{2}(p-2)\right]j^{-p'} \qquad \forall j\geq 3.
\]  
\end{mylem}
\begin{proof}
Since $p>2$, then $1<p'<2$. Let $r=p'-1$. In view of Lemma~\ref{fig}, we have
\[
    \operatorname{v}_p(t)\leq r \left[\sin_{p'}(\pi_{p'} t)\right]^{r-1}\leq r \big[\sin(\pi t)\big]^{r-1}
\]
and so
\begin{align}\label{equation42}
     \operatorname{z}_p(y) \leq \frac1\pi \arcsin\left[\left(
     \frac yr\right)^{\frac{1}{r-1}}\right]=:\operatorname{r}_p(y)
     \qquad \forall y\in[r,\infty).
\end{align}
Set
\[
    \eta(j):=r\sin\left(\frac{\pi}{2j} \right)^{r-1}.
\]
Then,
\[
    \operatorname{r}_p(\eta(j))=\frac{1}{2j}<\frac12.
\]
Here we use the requirement $j\geq 3$, in order to make sure that the arc-sine does not change branches.

Set
\[
    \operatorname{J}_1=\int_0^{\eta(j)}\mathrm{d} x
    = \eta(j)
\]
and
\[
    \operatorname{J}_2=\int_{\eta(j)}^\infty 
    \sin\left(j\pi\operatorname{r}_p(y)
    \right) \mathrm{d} y .
\]
Then,  \eqref{equation41} yields
\[
     |b_j|\leq \frac{4\pi_{p'}}{j^2\pi^2}(
     \operatorname{J}_1+\operatorname{J}_2).
\]
Here $\operatorname{J}_2$ is guaranteed to be on the right hand side,
because 
\[
     0<j\pi\operatorname{z}_p(y)\leq j\pi\operatorname{z}_p(\eta(j))\leq j\pi\operatorname{r}_p(\eta(j))=
     \frac \pi 2,
\]
so that $0<\sin(j\pi \operatorname{z}_p(y))\leq \sin(j\pi \operatorname{r}_p(y))$ for $y \in [\eta(j), \infty)$.

Let us estimate an upper bound for $\operatorname{J}_2$.
Changing variables to
\[
     t=j\pi\operatorname{r}_p(y) \Longleftrightarrow y=r\sin\left(\frac{t}{j}\right)^{r-1}
\]
gives
\begin{align*}
    \operatorname{J}_2&=\int_0^{\frac{\pi}{2}}
    \frac{r(1-r)}{j}\sin \left(\frac{t}{j}\right)^{r-2}
    \cos\left(\frac{t}{j}\right)\sin(t)\mathrm{d} t \\
    & = r(1-r)\int_0^{\frac{\pi}{2}}
    \sin \left(\frac{t}{j}\right)^{r-1}
    \left[\frac{\frac{t}{j}}{\sin \left(\frac{t}{j}\right)}   \right]
    \left(\frac{\sin t}{t}\right) 
    \cos\left(\frac{t}{j}\right)\mathrm{d} t
    .
\end{align*}
Note that,
\begin{align}\label{equation61}
    \max_{0< \theta \leq \frac \pi 2}\frac{\theta}{\sin \theta}=\frac
    \pi 2, \qquad
    \max_{0< \theta \leq \frac \pi 2}\frac{\sin \theta}{\theta}=1
\end{align}
and
\begin{align*}
0< t <j \pi\operatorname{r}_p(\eta(j)) = \frac\pi 2. 
\end{align*}
Here we are using once again the fact that $j\geq 3$. Then
\[
\operatorname{J}_2< \frac{\pi}{2}r(1-r)
\int_0^{\frac{\pi}{2}}
    \sin \left(\frac{t}{j}\right)^{r-1}
    \cos\left(\frac{t}{j}\right)\mathrm{d} t .
\]
Changing variables to
\[
    \tau=\sin \left(\frac{t}{j}\right),
\]
yields
\[
    \operatorname{J}_2 < \frac{j \pi}{2}r(1-r)\int_{0}^{\sin \frac{\pi}{2j}} \tau^{r-1} \mathrm{d} \tau=
    \frac{j\pi}{2}(1-r)\sin\left(\frac{\pi}{2j}\right)^r.
\]

Then
\[
    |b_j| < \frac{2\pi_{p'}}{j^2\pi^2}
    \left[2+\frac{j\pi(1-r)}{r}\sin\left(\frac{\pi}{2j}\right)\right]\eta(j).
\]
According to~\eqref{equation61},
we get
\[
    \eta(j)\leq rj^{1-r}
\]
and
\begin{align}\label{equation26}
   |b_j| < \frac{2\pi_{p'}r}{j^2\pi^2}
    \left[2+\frac{j\pi(1-r)}{r}\frac{\pi}{2j}\right]j^{1-r} .
\end{align}
Simplifying the expression on the right hand side, ensures the validity of the lemma.
\end{proof}


\section{The change of coordinates map} \label{sec4}

We now derive various properties of the change of coordinates maps that take the $2$-cosine functions into the $p$-cosine functions. Most of the material in this section can also be found in \cite{BBCDG2006}, \cite{BushellEdmunds2012}, \cite{EdmundsGurkaLang2014} and \cite{BL2014}. We keep a self-contained presentation here by including details of the main arguments.

Given any $g\in L_s$, denote the even extension of $g$ with respect to $1$ by
\[
 \tilde{g}(x)=\left\{\begin{array}{ll} g(x) & x \in [0, 1]\\
                          g(2-x)& x \in (1, 2].
                          \end{array}\right.
\]
A $2-$periodic extension of $g$ to the whole of $\mathbb{R}$ is then written as
\[
      g^*(x)=\tilde{g}(x-2\left\lfloor \frac x2 \right\rfloor). 
\]
The floor function $\lfloor y \rfloor \in \mathbb{Z}$ is the unique integer such that $y-\lfloor y \rfloor \in [0, 1)$.
For any $n\in \mathbb{N}$, let
\[
 M_ng(x):=g^*(nx).
\]

\begin{mylem}\label{parsnip}
The operators $M_n:L_s \longrightarrow L_s$ are linear isometries.
 \end{mylem}
\begin{proof} Indeed,
\begin{align*}
     \Vert M_n g\Vert_{L_s}^s&=\int_0^1 \vert M_n g(x) \vert^s \mathrm{d}x=\int_0^1 \vert g^*(nx)\vert^s \mathrm{d}x=\int_0^1 \vert \tilde g(nx-2\left\lfloor \frac{nx}{2}\right\rfloor) \vert^s \mathrm{d}x\\
                        &=\frac1n\int_0^n \vert \tilde g(y-2\left\lfloor \frac{y}{2}\right\rfloor) \vert^s \mathrm{d}y=\frac1n  \sum_{l=0}^{n-1} \int_l^{l+1} \vert \tilde g(y-2\left\lfloor \frac{y}{2}\right\rfloor) \vert^s \mathrm{d}y\\
                        &=\frac{1}{n} \left[ \sum_{\substack{
   l=0\\
   l\equiv_2 0
  }}^{n-1} \int_l^{l+1}  \vert \tilde g(y-2\left\lfloor \frac{y}{2}\right\rfloor) \vert^s \mathrm{d}y+\sum_{\substack{
   l=1\\
   l\equiv_2 1
  }}^{n-1} \int_l^{l+1}  \vert \tilde g(y-2\left\lfloor \frac{y}{2}\right\rfloor) \vert^s \mathrm{d}y \right].
  \end{align*}
Changing variables to $w=y-l$ for $l\equiv_2 0$ and $z=y-(l-1)$ for $l\equiv_2 1$, gives
\begin{gather*}
\left\lfloor \frac{y}{2}\right\rfloor=\begin{cases}\frac{l}{2} & \text{whenever } l\equiv_2 0 \\
\frac{l-1}{2}& \text{whenever } l\equiv_2 1.\end{cases}
\end{gather*}
Hence,
\begin{align*}
  \Vert M_n g\Vert_{L_s}^s&=\frac{1}{n} \left[ \sum_{\substack{
   l=0\\
   l\equiv_2 0
  }}^{n-1} \int_0^1  \vert g(w) \vert^s \mathrm{d}w+\sum_{\substack{
   l=1\\
   l\equiv_2 1
  }}^{n-1} \int_1^2  \vert \tilde g(z) \vert^s \mathrm{d}z\right].
\end{align*}
Another change of variables $z=2-w$, then yields
\[
 \Vert M_ng \Vert_{L_s}^s=\frac{1}{n}\left[n \int_0^1 \vert g(w)\vert^s \mathrm{d}w\right]=\Vert g \Vert_{L_s}^s
\]
as claimed.
\end{proof}
Let
$
    e_n(x):=\cos(n\pi x). 
$
If 
\[
      g=\frac{\widehat{g}(0)}{2} e_0+\sum_{j=1}^\infty \widehat{g}(j) e_j\in L_s    
\]
where 
\begin{align*}
\widehat{g}(k):=2\int_0^1 g(x) e_k(x) \mathrm{d}x \qquad \forall k\in \mathbb{N}\cup\{0\}
\end{align*}
are the corresponding cosine Fourier coefficients, 
then
\[
      M_n g=\frac{\widehat{g}(0)}{2} e_0+\sum_{j=1}^\infty \widehat{g}(j) M_n e_j
              =\frac{\widehat{g}(0)}{2} e_0+\sum_{j=1}^\infty \widehat{g}(j) e_{nj}\in L_s.
\]

Now, let $f_n(x):=\cos_p(n\pi_p  x)$. Note that $e_0(x)=f_0(x)=1$ for all $x\in \mathbb{R}$.
Suitable linear extensions of the map $A:e_n\mapsto f_n$ are the changes of coordinates between $\{e_n\}_{n=0}^\infty$ and $\{f_n\}_{n=0}^\infty$. Our next goal is to find a canonical decomposition for $A$ 
in terms of $M_n$ and the Fourier coefficients  $b_n(p)$. After that, we show that these are bounded operators of the Banach spaces $L_s$ for all $s>1$.  

\begin{pro} \label{sumpowerone}
For all $p>1$,
\[
   \sum_{j=1}^\infty |b_j(p)|<\infty.
\]
\end{pro}
\begin{proof}
This is a direct consequence of lemmas~\ref{fourcoeffboundless2} and \ref{tomato}. See \eqref{equation14} and \eqref{equation19} below.
\end{proof}

In the notation of Section~\ref{sec3}, we have $\widehat{f}_1(k)=b_k(p)$ 
for all $k\in \mathbb{N}\cup\{0\}$. Recall that $b_k=0$ for $k\equiv_2 0$. Since any of the functions $f_n(x)$ is continuous, then they all have a Fourier cosine expansion
\begin{align*}
f_n(x)=\frac12 \widehat{f}_n(0)e_0(x)+\sum_{k=1}^\infty \widehat{f}_n(k) e_k(x)
\end{align*}
which is both pointwise convergent for all $x\in[0,1]$ and also convergent in the norm of $L_s$ for all $s>1$. Then, for all $n>1$,
\begin{align*}
\widehat{f}_n(k)&=2\int_0^1f_1(nx) \cos(k\pi x)\mathrm{d}x\\
&=2\int_0^1\Big(\sum_{m=1}^\infty \widehat{f}_1(m) \cos(m\pi nx)\Big)\cos(k\pi x) \mathrm{d}x \\
&=2\sum_{m=1}^\infty \widehat{f}_1(m)\int_0^1\cos(mn\pi x)\cos(k\pi x)\mathrm{d}x\\
&=\begin{cases} b_m(p)&\text{for } mn=k,\ m\equiv_2 1\\
0 & \text{otherwise}. \end{cases}
\end{align*}
Here we can exchange the infinite summation with the integral sign, due to 
the pointwise convergence of the series, Proposition~\ref{sumpowerone} and the Dominated Convergence theorem.

Let 
\begin{align}\label{equation7}
 A:=\sum_{j=1}^\infty b_j(p) M_j.
\end{align}
By virtue of Proposition~\ref{sumpowerone}, Lemma~\ref{parsnip} and the triangle inequality, it follows that the expression \eqref{equation7} is convergent in the operator norm of $L_s$ and that $A:L_s\longrightarrow L_s$ is a bounded linear operator such that
\begin{align*}
\Vert A \Vert_{L_s\longrightarrow L_s} \leq \sum_{j=1}^\infty \vert b_j \vert \Vert M_j \Vert_{L_s\longrightarrow L_s}=\sum_{j=1}^\infty \vert b_j \vert.
\end{align*}
Moreover,
\[
    Ae_0=\sum_{j=1}^\infty b_j M_j e_0=\sum_{j=1}^\infty b_j e_0=
    \sum_{j=1}^\infty b_j e_j(0)=\cos_p(\pi_p 0)=1=f_0
\]
and
\begin{align*}
 Ae_n=\sum_{j=1}^\infty b_j M_je_n=\sum_{j=1}^\infty \widehat{f}_1(j) e_{nj}=\sum_{k=1}^\infty \widehat{f}_n(k) e_k=f_n \qquad \forall n\in \mathbb{N}.
\end{align*}
These are the change of basis maps between $\{e_n\}_{n=0}^\infty$ and 
$\{f_n\}_{n=0}^\infty$.

The operator $A$ is an homeomorphism of $L_s$ if and only if the family $\{\cos_p(n\pi_p  \cdot)\}_{n=0}^\infty$ is a Schauder basis of $L_s$,
cf. \cite{Higgins1977} or \cite{Singer1970}. Then we have the following criterion, which is a consequence of \cite[Theorem IV-1.16]{Kato1967},
\begin{equation}   \label{herb} 
 \sum_{\substack{
   j=3\\
   j\equiv_2 1
  }}^\infty \vert b_j(p) \vert<\vert b_1(p) \vert
  \quad \Rightarrow \quad  \begin{cases}
  \{\cos_p(n\pi_p  \cdot)\}_{n=0}^\infty \text{ is a Schauder} \\\text{basis of } L_s \text{ for all }s >1.
\end{cases}
\end{equation}
We employ this criterion below in order to determine the basis thresholds for the family $\{\cos_p(n\pi_p \cdot)\}_{n=0}^\infty$
claimed in Theorem~\ref{melon}.


\section{Proof of Theorem~\ref{melon}}
 \label{sec5}
 
The proof is separated into two cases. 

\subsection{The case $1<p<2$}
Recall the expression for $c_p$ given in \eqref{defcp} and
consider the identity
\begin{equation}   \label{equcasep<2}
\pi_{p}^2 c_{p}=\frac{\pi^3}{\pi^2-8}.
\end{equation}

\begin{mylem}\label{lemma1} There exists $1<p_0<2$ such that
\eqref{equcasep<2} holds true for $p=p_0$. Moreover, 
\[
\pi_{p}^2c_p < \frac{\pi^3}{\pi^2-8} \qquad \forall p\in(p_0,2).
\]
\end{mylem}
\begin{proof}
It will be enough to prove that 
$\pi_p^2 c_p$ is a convex function of the parameter $p$ for all $1<p<2$. Indeed, since
\[
\lim_{p \rightarrow 1^+}\pi_p^2 c_p = \infty \qquad \text{and} \qquad
 \lim_{p \rightarrow 2^-}\pi_p^2 c_p=\pi^2< \frac{\pi^3}{\pi^2-8}, 
\] 
both statements will immediately follow from this property.

Firstly note that
\[
 \frac{\mathrm{d}}{\mathrm{d}p}\ln(p-1)^\frac{p-1}{p}=\frac{1}{p^2}\ln(p-1)+\frac{1}{p}
\]
and
\begin{align*}
 \frac{\mathrm{d}^2}{\mathrm{d}p^2}\ln(p-1)^\frac{p-1}{p}=\frac{2-p}{p^2(p-1)}-2\frac{\ln(p-1)}{p^3}>0.
\end{align*}
Then $\ln(p-1)^\frac{p-1}{p}$ is convex for $1<p<2$. 

Similarly, we have
\[
 \frac{\mathrm{d}}{\mathrm{d}p}\ln(2-p)^\frac{2-p}{p}=\frac{-2}{p^2}\ln(2-p)-\frac{1}{p}
\]
and
\begin{align*}
 \frac{\mathrm{d}^2}{\mathrm{d}p^2}\ln(2-p)^\frac{2-p}{p}=\frac{4-p}{p^2(2-p)}+4\frac{\ln(2-p)}{p^3}>0.
\end{align*}
Then, also $\ln(2-p)^\frac{2-p}{p}$ is convex for $1<p<2$.

Furthermore,
\[
 \frac{\mathrm{d}}{\mathrm{d}p}[\ln \pi_p]= \frac{\pi\cot(\frac{\pi}{p})}{p^2}-\frac{1}{p}
\]
and
\[
\frac{\mathrm{d^2}}{\mathrm{d}p^2}\ln \pi_p=\frac{(p^2+\pi^2)}{p^4}-\frac{2 \pi}{p^3}\cot\Big(\frac{\pi}{p}\Big)+\frac{\pi^2}{p^4} \cot^2\Big(\frac{\pi}{p}\Big)>0.
\]
The latter is a consequence of the fact that $\cos \frac{\pi}{p}<0$ and $\sin\frac{\pi}{p}>0$. Hence, also $\ln \pi_p^2$ is convex for $1<p<2$. 

The convexity of  the logarithm of each one of the multiplying terms in the expression for $\pi_p^2 c_p$, implies that $\ln \pi_p^2 c_p$ is convex for $1<p<2$.  This ensures that indeed  $\pi_p^2 c_p$ is convex in the same segment and the validity of the statement is ensured. 
\end{proof}

\begin{cor} \label{apple} Let $1<p_0<2$ be such that \eqref{equcasep<2} holds true for $p=p_0$. The family  $\{\cos_p(n\pi_p  \cdot)\}_{n=0}^\infty$ is a Schauder basis of $L_s$ for all $s>1$ and $p_0\leq p \leq 2$.
\end{cor}
\begin{proof}
According to Lemma~\ref{fourcoeffboundless2},
\begin{equation}\label{equation14}
\sum_{\substack{
   j=3\\
   j\equiv_2 1
  }}^\infty \vert b_j(p) \vert
	<\frac{8\pi_pc_p}{\pi^2}\sum_{\substack{
   j =3\\
   j\equiv_2 1
  }}^\infty \frac{1}{j^2}
	=\frac{\pi_p^2 c_p(\pi^2-8)}{\pi^2 \pi_p}.
\end{equation}
On the other hand, in view of Lemma~\ref{plum} and Lemma~\ref{fig}-\ref{figa}, we have
\begin{align*}
 b_1(p) &=\frac{\pi}{\pi_p}a_1=\frac{4\pi}{\pi_p}\int_0^\frac{1}{2}\sin_p(\pi_px)\sin(\pi x)\mathrm{d}x\\
&\geq \frac{4\pi}{\pi_p}\int_0^\frac{1}{2} \sin(\pi x)^2\mathrm{d}x=\frac{\pi}{\pi_p}.
\end{align*}
Then, Lemma~\ref{lemma1} yields
\[
\sum_{\substack{
   j=3\\
   j\equiv_2 1
  }}^\infty \vert b_j(p) \vert <b_1(p)
\]
for all $p \in [p_0, 2)$. By virtue of \eqref{herb} the claimed conclusion follows.
\end{proof}

Since
\[
 \pi_{\frac43}^2 c_{\frac43}=\frac{\pi^2 3^{\frac54}\sqrt{2}}{2}>\frac{\pi^3}{\pi^2-8}
\]
and
\[
 \pi_{\frac32}^2 c_{\frac32}=\frac{64\pi^2 }{27\sqrt[3]{4}}<\frac{\pi^3}{\pi^2-8},
\]
then $\frac43<p_0<\frac32$. This settles the proof of Theorem~\ref{melon} for $1<p<2$. 

\begin{myrem}An implementation of the Newton method gives $p_0 \approx 1.458801$ as an approximated solution of \eqref{equcasep<2} with all digits correct.\end{myrem}

\subsection{Case $p>2$}
Recall the following identities involving the Riemann Zeta function \cite[3.411, 9.522 \& 9.524]{Grad2007},
\begin{equation} \label{intrepzeta}
\zeta(q)=\frac{1}{\Gamma(q)}\int_0^\infty \frac{t^{q-1}}{e^t-1} \mathrm{d}t \qquad \qquad  \mathrm{Re}(q)>1,
\end{equation}
\begin{align}\label{equation63}
     \sum_{\substack{ j=1 \\ j\not \equiv_2 0}}^\infty
     \frac{1}{j^{q}}=\left(1-\frac{1}{2^{q}} \right) \zeta(q)
\end{align}
and
\begin{align}\label{equation36}
\frac{\zeta'(q)}{\zeta(q)}=-\sum_{k=1}^\infty \frac{\Delta(k)}{k^q}
\end{align}
where
\[
    \Delta(k)=\begin{cases} \ln(r) & \text{if }k=r^m \text{ for some }r\text{ prime and }m\in \mathbb{N} \\
    0 &   \text{otherwise}.      \end{cases}
\]

\begin{mylem} \label{estzeta3/2}
Let 
\[
    t_0=\frac{2(e^2-3e+1)}{(e^2-2e-1)}.
\]
Then
\begin{equation} \label{c1}
    \zeta\left(\frac32\right)<\frac{2}{\sqrt{\pi}}\left(2\sqrt{2} \arctan \frac{1}{\sqrt{2}} +\frac{\pi^2}{6} + \frac{t_0^2}{4}-\frac{(t_0-1)^2}{2(e-1)^2}
    -\frac{t_0(e-2)+1}{e-1}\right).
\end{equation}
\end{mylem}
\begin{proof}
Since $\Gamma(1+\frac12)=\frac{\sqrt{\pi}}{2}1!!=\frac{\sqrt{\pi}}{2}$, the representation \eqref{intrepzeta} gives
\begin{align*} 
    \zeta\left(\frac32\right)&=\frac{2}{\sqrt{\pi}}\int_{0}^\infty
    \frac{t^{1/2}}{e^t-1} \mathrm{d} t\\
    &=\frac{2}{\sqrt{\pi}}\left(\int_{0}^1+\int_{1}^\infty
    \frac{t^{1/2}}{e^t-1} \mathrm{d} t\right)=\frac{2}{\sqrt{\pi}}(\mathrm{J}_1+\mathrm{J}_2).
\end{align*}
We estimate separately upper bounds for $\mathrm{J}_1$ and $\mathrm{J}_2$.

The change of variables $t=u^2$, yields
\begin{align*}
\mathrm{J}_1&=\int_0^1 \frac{t^{1/2}}{e^t-1} \mathrm{d} t 
< \int_0^1 \frac{t^{1/2}}{t+\frac{t^2}{2}} \mathrm{d} t \\
&=\int_0^1 \frac{2u^2}{u^2+\frac{u^4}{2}} \mathrm{d} u=
2\sqrt{2} \arctan \frac{1}{\sqrt{2}}.
\end{align*}

On the other hand, we know that $\zeta(2)=\int_0^\infty \frac{t}{e^t-1}\mathrm{d}t=\frac{\pi^2}{6}$, so
\[
    \mathrm{J}_2\leq \int_1^\infty \frac{t}{e^t-1} \mathrm{d} t
    =\frac{\pi^2}{6}-\int_0^1 \frac{t}{e^t-1} \mathrm{d} t.
\]
We find lower bound for the integral on the right hand side, by interpolating the curve $c(t)=\frac{t}{e^t-1}$ at two points, $t=0$ and $t=1$. Firstly observe that $c(t)\to 1$ as $t\to 0$, $c(t)$ is decreasing  and $c''(t)\geq 0$ for $t\in [0,1]$. Let $t_0$ be as in the hypothesis and let
\[
     \tilde{c}(t)=\begin{cases} 1-\frac12 t & 0\leq t \leq t_0 \\
     \frac{1}{(e-1)^2}(1-t) +\frac{1}{e-1} & t_0\leq t \leq 1    \end{cases}
\]
be the piecewise linear interpolant of $c(t)$ in the two segments
$[0,t_0]$ and $[t_0,1]$, which is continuous at $t_0$. Note that $\tilde{c}(t)$ and $c(t)$ are tangent at $t=0$ and $t=1$. Then
\[
      c(t)\geq \tilde{c}(t) \qquad \qquad \forall t\in[0,1].
\]
Hence
\begin{align*}
\int_0^1 c(t) \mathrm{d} t & \geq 
\int_{0}^{t_0} \left(1-\frac12 t \right) \mathrm{d} t +
\int_{t_0}^1 \left(\frac{1}{(e-1)^2}(1-t) +\frac{1}{e-1}\right)
\mathrm{d} t \\
&= -\frac{t_0^2}{4} +\frac{(t_0-1)^2}{2(e-1)^2}+\frac{t_0(e-2)+1}{e-1}.
\end{align*}
Thus 
\[
    \mathrm{J}_2\leq \frac{\pi^2}{6}+ \frac{t_0^2}{4}-\frac{(t_0-1)^2}{2(e-1)^2}    -\frac{t_0(e-2)+1}{e-1}.
\]
Alongside with the upper bound above for $\mathrm{J}_1$, this ensures the validity of the claimed statement.
\end{proof}

Now, consider the equation
\begin{equation}\label{equation16}
 \frac{2\pi_{p'}}{\pi^2(p-1)}\left[2+\frac{\pi^2}{2}(p-2)\right] \left[\left(1-\frac{1}{2^{p'}}\right)\zeta(p')-1\right]=\frac{8}{\pi \pi_p}.
\end{equation}

\begin{mylem}\label{lemma2} 
There exists $p_1\in(\frac{11}{5},3)$ such that 
\eqref{equation16} holds true for $p=p_1$.
Moreover,
\[
 \frac{2\pi_{p'}}{\pi^2(p-1)}\left[2+\frac{\pi^2}{2}(p-2)\right] \left[\left(1-\frac{1}{2^{p'}}\right)\zeta(p')-1\right]<\frac{8}{\pi \pi_p} \qquad
\forall p\in[2,p_1).
\]
\end{mylem}
\begin{proof}
From \eqref{equation62} it follows that the identity \eqref{equation16} reduces to
\begin{equation}   \label{sameeq16}
 \frac{\pi}{p^2\sin(\frac{\pi}{p})^2}\left(2+\frac{\pi^2}{2}(p-2)\right) \left[\left(1-\frac{1}{2^{p'}}\right)\zeta(p')-1\right]=1.
\end{equation}
Denote by $h(p)$ the left hand side of \eqref{sameeq16}. Then $h:(1,\infty)\longrightarrow \mathbb{R}$ is continuous and
\[
h(2)=\frac{\pi}{2}\left(\frac{\pi^2}{8}-1\right)<1.
\]
Since
\[
      \zeta\left(\frac32\right)>1+\frac{\sqrt{2}}{4}+\sqrt{3} \sum_{k=3}^\infty
     \frac{1}{k^2} =\frac{4+\sqrt{2}}{4}+\sqrt{3} \left(\frac{\pi^2}{6}-\frac54\right),
\]
we get
\begin{align*}
 h&(3)=\frac{\pi}{9\sin(\frac{\pi}{3})^2}\left[2+\frac{\pi^2}{2}\right] \left[\left(1-\frac{1}{2^{\frac32}}\right)\zeta\left(\frac32\right)-1\right]\\
&>\frac{\pi}{9\sin(\frac{\pi}{3})^2}\left[2+\frac{\pi^2}{2}\right] \left[\left(1-\frac{1}{2^{\frac32}}\right)\left( \frac{4+\sqrt{2}}{4}+\sqrt{3} \left(\frac{\pi^2}{6}-\frac54\right) \right)-1\right] \\
   &>1.
\end{align*}
Hence, there exists $p_1\in(2,3)$ such that $h(p_1)=1$.

The derivative
\[
\frac{\mathrm{d} }{\mathrm{d} q}\left[\left(1-\frac{1}{2^q}\right)\zeta(q)\right]=\frac{\ln(2)}{2^q}\zeta(q)+\left(1-\frac{1}{2^q}\right)\zeta'(q)
\]
is negative for any $q \in (1, 2)$. Indeed the identity \eqref{equation36} gives
\begin{align*}
\frac{\zeta'(q)}{\zeta(q)}&<-\frac{\ln(2)}{2^q}-\frac{\ln(3)}{3^q}-\frac{\ln(2)}{4^q}\\
&<-\ln(2)\left[\frac{1}{2^q}+\frac{1}{3^q}+\frac{1}{4^q}\right]<\frac{\ln(2)}{1-2^q },
\end{align*}
so that
\[
\frac{\mathrm{d} }{\mathrm{d} q}\left[\left(1-\frac{1}{2^q}\right)\zeta(q)\right]
=\zeta(q)\left[\frac{\ln(2)}{2^q}+\frac{2^q-1}{2^q}\frac{\zeta'(q)}{\zeta(q)}\right]<0.
\]
Since $p'$ and $\sin\big(\frac{\pi}{p}\big)$ are decreasing functions of $p>2$,
then
\[
 \frac{\pi}{\sin(\frac{\pi}{p})^2}\left[\left(1-\frac{1}{2^{p'}}\right)\zeta(p')-1\right]
\]
is an increasing function of $p>2$.

As
\[
 \frac{\mathrm{d} }{\mathrm{d} p}\left[\frac{1}{p^2}\Big(2+\frac{\pi^2}{2}
(p-2)\Big)\right]=\frac{1}{p^3}(-\frac{\pi^2}{2}p+2\pi^2-4)>0 \qquad
\forall p \in [2, 3],
\]
then $h(p)$ is increasing for $p \in [2, 3]$ and so indeed
\[
 h(p)<h(p_1)=1 \qquad \forall p \in [2, p_1).
\]

Let us now show that $p_1>\frac{11}{5}$. Let $c_1$ denote the right hand side of the estimate \eqref{c1} in Lemma~\ref{estzeta3/2}. Since $\zeta(q)$ is convex in the segment $[\frac32, 2]$, then
\[
     \zeta (q)\leq \left(\frac{\pi^2}{3}-2c_1\right)(q-2)+\frac{\pi^2}{6}.
\]
That is, the straight line joining the points $(\frac32,c_1)$ and $(2,\frac{\pi^2}{6})$ is above the curve $\zeta(q)$ for all $q\in [\frac32, 2]$. 
Then
\begin{equation}    \label{estforh1}
      \zeta\left(\frac{11}{6}\right)\leq \frac{\pi^2}{9}+\frac{c_1}{3}.
\end{equation}
Note that for $p=\frac{11}{5}$, $p'=\frac{11}{6}$. Now, $\sin(\pi y)$ is concave for
$y\in[\frac{5}{12},\frac12]$. Then it is above the straight line joining the points
$(\frac{5}{12},\sin \frac{5\pi}{12})$ and $(\frac12,1)$. That is  
\[
   \sin\left(\pi y\right) \geq \left(12-12 \sin\frac{5\pi}{12}\right)\left(y-\frac12\right)+1 \qquad \forall y\in\left[\frac{5}{12},\frac12\right]. 
\]
Then
\begin{equation}   \label{estforh2}
         \sin\frac{5\pi}{11} > \frac{  \sqrt{6}}{22} {\left(\sqrt{3} + 3\right)} + \frac{5}{11}.
\end{equation}
Denote by $c_2$ the right hand side of the latter inequality.
From \eqref{estforh1} and \eqref{estforh2}, it follows that
\begin{align*}
     h\left(\frac{11}{5}\right)& = \frac{\pi}{(\frac{11}{5})^2\sin(\frac{5\pi}{11})^2}\left[2+\frac{\pi^2}{2}\left(\frac{11}{5}-2\right)\right] \left[\left(1-\frac{1}{2^{11/6}}\right)\zeta\left(\frac{11}{6}\right)-1\right] \\
&< \frac{\pi}{\frac{121}{25}c_2^2}\left(2+\frac{\pi^2}{10}\right) \left[\left(1-\frac{1}{2^{11/6}}\right)\left(\frac{\pi^2}{9}+\frac{c_1}{3}\right)-1\right]<1.
\end{align*}
As $h(p)$ is increasing, then indeed $p_1>\frac{11}{5}$.
\end{proof}
\begin{cor}\label{berry}
 Let $p_1>2$ be such that \eqref{equation16} holds true for $p=p_1$.
The family $\{\cos_p(n\pi_p  \cdot \}_{n=0}^\infty$ forms a Schauder basis of $L_s$ for all $s>1$ and $2\leq p \leq p_1$.
\end{cor}
\begin{proof}
From Lemma~\ref{tomato} and \eqref{equation63}, we have
\begin{align}\label{equation19}
   \sum_{\substack{
   j=3\\
   j\equiv_2 1
  }}^\infty |b_j|< \frac{2\pi_{p'}}{\pi^2(p-1)}\left[2+\frac{\pi^2}{2}(p-2)\right]\left[\left(1-\frac{1}{2^{p'}} \right) \zeta(p')-1\right].
\end{align}
According to part ``\ref{lbcp2}'' of Lemma~\ref{basicpropcos}, $\sin_p(\pi_p x)$ is strictly concave on $(0, \frac12)$. Then
\begin{align*}
 a_1&=2\int_0^1 \sin_p(\pi_p x) \sin(\pi x) \mathrm{d}x=4\int_0^\frac12 \sin_p(\pi_p x)\sin(\pi x) \mathrm{d}x  \\
&>4\int_0^\frac12 (2x) \sin(\pi x) \mathrm{d}x=\frac{8}{\pi^2}.
\end{align*}
Hence, in view of Lemma~\ref{plum}, we get
\begin{align}\label{equation25}
b_1=\frac{\pi}{\pi_p}  a_1 >\frac{8}{\pi \pi_p}.
\end{align}
From Lemma~\ref{lemma2}, it then follows that
\[
\sum_{\substack{
   j = 3\\
   j\equiv_2 1
  }}^\infty \vert b_j(p) \vert < b_1(p) \qquad \forall p\in[2,p_1].
\]
By virtue of \eqref{herb}
 this implies the claimed conclusion. 
\end{proof}

\begin{myrem}
An approximation of the solution of \eqref{equation16} via the Newton Method gives $p_1\approx 2.42865$ with all digits correct.
\end{myrem}


\section{Connections with other work} \label{sec6}
In this final section we describe various connections between the statements established above and those reported in the literature.   

\subsection*{The $p$-exponential functions}
Let
\[
      \exp_p(iy)=\cos_p(y)+i\sin_p(y) \qquad \qquad \forall y\in \mathbb{R}.
\]
By combining Theorem~\ref{melon} with  \cite[Theorem~1]{BBCDG2006}
or \cite[Theorem~4.5]{BushellEdmunds2012}, it immediately follows that
the family $\tilde{\mathcal{F}}=\{\exp_p(i n\pi_p\cdot)\}_{n=-\infty}^{\infty}$ is a Schauder basis of the Banach space $L_s(-1,1)$ for all $p\in [p_0,p_1]$.

Indeed, recall that every $f\in L^s(-1,1)$ decomposes as $f=f_{\mathrm{e}}+ f_{\mathrm{o}}$ for
\[
      f_{\mathrm{e}}(x)=\frac{f(x)+f(-x)}{2} \qquad\text{and} \qquad
      f_{\mathrm{o}}(x)=\frac{f(x)-f(-x)}{2},
\]
the even and odd parts of $f$, respectively. The family $\{\cos_p(n\pi_p\cdot)\}_{n=0}^\infty$ comprises only even functions, the family $\{\sin_p(n\pi_p\cdot)\}_{n=1}^\infty$  comprises only odd functions and they are Schauder bases of the corresponding subspaces of $L_s(-1,1)$ for $p\in [p_0,p_1]$. This implies that there exist two unique scalar sequences $(\alpha_k)_{k=0}^\infty$ and $(\beta_k)_{k=1}^\infty$, such that
\[
      f(\cdot)=\alpha_0+\sum_{k=1}^\infty \alpha_k \cos_p(k\pi_p\cdot)+i\beta_k\sin_p(k\pi_p\cdot)
\]  
in $L_s(-1,1)$. In order to see this, one expands $f_{\mathrm{e}}$ in $\{\cos_p(n\pi_p\cdot)\}_{n=0}^\infty$ and $f_{\mathrm{o}}$ in $\{\sin_p(n\pi_p\cdot)\}_{n=1}^\infty$, in the corresponding even and odd subspaces.

By letting $c_0=\alpha_0$, 
\[
      c_k=\frac{\alpha_k+\beta_k}{2} \qquad \text{and} \qquad
      c_{-k}=\frac{\alpha_k-\beta_k}{2} \qquad \forall k\in \mathbb{N},
\] 
we get
\[
   f(\cdot)=\sum_{k=-\infty}^\infty c_k \exp_p(i k\pi_p\cdot)
\]
in $L_s(-1,1)$. Since there is a 1:1 correspondence between the scalar sequences via
\[
      \alpha_k=c_k+c_{-k} \qquad \text{and} \qquad
      \beta_{k}=c_k-c_{-k},
\] 
then in fact $(c_k)_{k=-\infty}^\infty$ is unique for the given $f$. Thus, 
$\tilde{\mathcal{F}}$ satisfies the definition of a Schauder basis for the Banach space $L_s(-1,1)$.

\subsection*{The regularity of the $p$-sine functions}
Let $r> 0$ and denote by $H^r\equiv H^r(0,1)$ the (Hilbert) Sobolev space of order $r$. Let $1<p<2$. According to the formula \cite[(4.4)]{BushellEdmunds2012}, it follows that the Fourier coefficients of the $p$-sine function are such that
\begin{align*}
      \vert a_j(p) \vert \leq \frac{16 \pi_p^2c_p}{\pi^3}j^{-3} \qquad \forall j\in \mathbb{N}.
\end{align*}
Then, $\sin_p(\pi_p\cdot)\in H^{\rho}$ for all $\rho <\frac{5}{2}$.

Numerical estimates for the Sobolev regularity of $\sin_p(\pi_p\cdot)$ for $2<p<100$ were reported in \cite[Figure~2]{BL2010}. From that picture, one may conjecture that for $p>3$, $\sin_p(\pi_p\cdot) \notin H^2$. Moreover,  the regularity appears to drop asymptotically to $\frac32$ for $p$ large. By contrast,  it appears that $\sin_p(\pi_p\cdot) \in H^2$ for $2<p<3$. The following statement, which is a consequence of Lemma~\ref{tomato}, settles this conjecture.

\begin{cor}\label{raspberry}
For $p>2$ set $r(p)= p'+\frac1 2$. Then $\sin_p(\pi_p\cdot)\in H^{\rho}$ for all $0\leq \rho<r(p)$.
\end{cor}
\begin{proof}
According to Lemma~\ref{plum}, 
\[
 \vert a_j(p) \vert = \frac{\pi_p}{j \pi} \vert b_j(p) \vert.
\]
Then, by virtue of Lemma~\ref{tomato},
\[
 \vert a_j(p) \vert \leq \frac{2\pi_p \pi_{p'}}{\pi^3 (p-1)}\left[2+\frac{\pi^2}{2}(p-2)\right] j^{-(p'+1)} \qquad \forall j\geq 3.
\]
Let $\langle j \rangle^2=1+j^2$. For $\rho< p'+\frac1 2$, 
\[
 \sum_{j=1}^\infty \langle j \rangle^{2\rho}\vert a_j(p) \vert^2 \leq 2^\rho a_1(p)^2+c(p) \sum_{\substack{
   j = 3\\
   j\equiv_2 1
  }}^\infty \frac{1}{j^{1+\epsilon(p)}}< \infty 
\]
where
\[
c(p)=\frac{2\pi_p \pi_{p'}}{\pi^3 (p-1)}\left[2+\frac{\pi^2}{2}(p-2)\right] \quad
\text{and} \quad \epsilon(p)=1-2\rho+2p'>0.
\]
Hence $\sin_p(\pi_p \cdot) \in H^\rho$ as claimed.
\end{proof}

The recent paper \cite{EdmundsGurkaLang2014-2} includes various intriguing results connected to Corollary~\ref{raspberry}.

\subsection*{The paper  \cite{EdmundsGurkaLang2014}}
The recent paper \cite{EdmundsGurkaLang2014} seems to be the only one in the existing literature which conducts an analysis of the basis properties of the $p$-cosine functions. In the notation of \cite{EdmundsGurkaLang2014} we fix $\alpha=1$ and $p=q>1$. The Fourier coefficients of the $p$-cosine functions are
\[
      \tau_j(p,p,1)=b_j(p) \qquad \forall j\in \mathbb{N}\cup\{0\}.
\]
The condition \cite[(2.2)]{EdmundsGurkaLang2014} as well as the criterion for determining whether $\{\cos_p(n\pi_p \cdot)\}_{n=0}^\infty$ is a Schauder basis of $L^s$ are exactly 
the same as \eqref{herb}. Let us compare some of the results of \cite{EdmundsGurkaLang2014} with those of the present work. 

In \cite[Proposition 2.5]{EdmundsGurkaLang2014}, the estimate  \cite[(2.20)]{EdmundsGurkaLang2014} is equivalent to the following. There exists $p_0^*=\frac{72(\pi-2)-2\pi^3}{96(\pi-2)-3\pi^3}$, such that
\begin{equation}  \label{equation18}
\tau_1(p, p, 1) \geq \begin{cases} 
\frac{\pi(p-1)}{2p-1}-\frac{(\pi-2)(p-1)}{3p-2} & 1<p<p_0^*   \\
\frac{\pi(p-1)}{2p-1}-\frac{\pi^3(p-1)}{24(4p-3)} & p_0^*<p<\infty.
\end{cases}
\end{equation}
Here $p_0^*$ satisfies the identity 
\[
 \frac{4p-3}{3p-2}=\frac{\pi^3}{24(\pi-2)}.
\]
Note that $p_0^*\approx 1.22$.

Let us consider firstly the regime $1<p<2$. 
From \cite[Proposition~2.2]{EdmundsGurkaLang2014} it follows that
\begin{equation}  \label{egl1}
    \sum_{k=1}^\infty|\tau_{2k+1}(p,p,1)| \leq \frac{\pi_p(\pi^2-8)}{\pi^2} \qquad \forall p\in(1,2).
\end{equation}
As  $c_p<1$ whenever $1<p<2$ in \eqref{defcp}, then \eqref{equation14} is sharper than \eqref{egl1} in this regime.

If $1<p<p_0^*$, then
\[
\frac{\pi_p(\pi^2-8)}{\pi^2}>\frac{\pi(p-1)}{2p-1}-\frac{(\pi-2)(p-1)}{3p-2},
\]
and no conclusion about the validity of \eqref{herb} can be derived in this case from \eqref{equation18} and \eqref{egl1}.
For $p_0^*<p<2$,  on the other hand,
\[
\frac{\pi_p(\pi^2-8)}{\pi^3}<\frac{p-1}{2p-1}-\frac{\pi^2(p-1)}{24(4p-3)}\quad \iff \quad p\in(p_0^\dagger,2),
\]
where $p_0^\dagger\approx 1.75$. In order to see this, note that $\pi_p$ is decreasing and $\lim_{p\to 1^{+}}\pi_p=\infty$, while the right hand side of this identity is increasing for $1<p<2$. Thus, a combination of 
\cite[Proposition~2.2]{EdmundsGurkaLang2014} and 
\cite[Proposition~2.5]{EdmundsGurkaLang2014}, only guarantees that 
$\{\cos_p(n\pi_p \cdot)\}_{n=0}^\infty$ is a Schauder basis of $L^s$ for  $p \in[p_0^\dagger, 2)$ where $p_0^\dagger>\frac32>p_0$.

As it turns, it is not possible to deduce from the results of \cite{EdmundsGurkaLang2014} any basis property of the family $\{\cos_p(n\pi_p \cdot)\}_{n=0}^\infty$ in the complementary regime $p>2$. Here is how the different estimates on the Fourier coefficients compare in this case.

From \cite[Proposition 2.4]{EdmundsGurkaLang2014}, we gather that 
\begin{equation} \label{egl2}
\sum_{k=1}^\infty \vert \tau_{2k+1}(p,p,1) \vert\leq \frac{2\pi_{p'}}{\pi^2(p-1)}\left[4+\pi(p-1)\right]\left[\left(1-\frac{1}{2^{p'}} \right) \zeta(p')-1\right].
\end{equation}
Since 
\[
       4+\pi(p-1)\geq 2+\frac{\pi^2}{2}(p-2) \qquad \forall p\leq
       \frac{4+2\pi^2-2\pi}{\pi^2-2\pi},
\]
the upper bound \eqref{equation19} is sharper than \eqref{egl2}
for $2\leq p \leq 3$. The latter is the relevant regime in the 
proof of Theorem~\ref{melon}. 
 
Since $\pi_p<\pi$ for $p>2$, the lower bound \eqref{equation25} is sharper than \cite[(2.19)]{EdmundsGurkaLang2014}. Moreover, 
\[
       \frac{8}{\pi \pi_p}> \frac{\pi(p-1)}{2p-1}-\frac{\pi^3(p-1)}{24(4p-3)} \qquad \qquad \forall p>2.
\]
Hence the estimate \eqref{equation18}, which is \cite[(2.20)]{EdmundsGurkaLang2014}, is also superseded by \eqref{equation25} for
$p>2$.


\section*{Acknowledgements}
HM was supported by Heriot-Watt University under the Global Platform Scholarships Scheme.

\subsection*{2010 Mathematics Subject Classification} Primary: 42A16. Secondary: 42A65. 

\subsection*{Keywords} Generalised cosine functions, Schauder basis properties, $p$-Laplacian.

\end{document}